\documentclass{amsproc}
\usepackage{amsmath, euscript}
\usepackage{cases}
\usepackage{mathrsfs}
\usepackage{bbm}
\usepackage{amssymb}
\usepackage{txfonts}
\usepackage{amscd}
\usepackage{amsfonts,latexsym,amsmath,amsthm,amsxtra,mathdots,amssymb,latexsym,mathabx}
\usepackage[all,cmtip]{xy}
\usepackage{color}
\usepackage{multicol}
\usepackage{hyperref}
\usepackage{tikz}

\allowdisplaybreaks

\def\mod{\mathrm{mod}\ }

    \newcommand{\BC}{{\mathbb {C}}} 
     
     \newcommand{\BH}{{\mathbb {H}}}

    \newcommand{\BQ}{{\mathbb {Q}}} \newcommand{\BR}{{\mathbb {R}}}

     \newcommand{\BZ}{{\mathbb {Z}}}

    \newcommand{\GL}{{\mathrm{GL}}}

    \newcommand{\PGL}{{\mathrm{PGL}}} 
    \renewcommand{\Re}{{\mathfrak{Re}\,}}
    
    \newcommand{\PSL}{{\mathrm{PSL}}}

\def\fra{\mathfrak{a}}
\def\frb{\mathfrak{b}}

\def\-{^{-1}}

\def\BCx{\mathbb C^\times}

\def\-{^{-1}}

\def\lp {\left (}
\def\rp {\right )}

\def\EJ{\EuScript J}

\def\Voronoi{Vorono\" \i \hskip 3 pt}

\newcommand{\delete}[1]{}

    \newcommand{\SL}{{\mathrm{SL}}}
    \newcommand{\SO}{{\mathrm{SO}}}
    \newcommand{\SU}{{\mathrm{SU}}}
     \newcommand{\Tr}{{\mathrm{Tr}}}

    \newcommand{\ds}{\displaystyle}
    \newcommand{\sstyle}{\scriptstyle}

    \newcommand{\ra}{\rightarrow}

    \theoremstyle{plain}

    \newtheorem{thm}{Theorem}[section] 
    \newtheorem{lem}[thm]{Lemma}

    \newtheorem {rem}[thm]{Remark}

    \numberwithin{equation}{section}
    
    \newtheorem{acknowledgement}{Acknowledgements}

\begin{document}

\title{On the Kuznetsov Trace Formula for $\PGL_2(\BC)$}

\author{ Zhi Qi}


%\subjclass[2010]{44A20, 33E20}
%\keywords{Hankel transforms, fundamental Bessel kernels}

\address{Department of Mathematics\\ Rutgers University\\Hill Center\\110 Frelinghuysen Road\\Piscataway, NJ 08854-8019\\USA}
\email{zhi.qi@rutgers.edu}

\begin{abstract}
In this note, using a representation theoretic method of Cogdell and Piatetski-Shapiro, we prove the  Kuznetsov trace formula for an {arbitrary} discrete group $\Gamma$ in $ \PGL_2 (\BC)$ that is cofinite but not cocompact. An essential ingredient is a kernel formula, recently proved by the author, on Bessel functions for $ \PGL_2 (\BC)$. This approach avoids the difficult analysis in the existing method due to Bruggeman and Motohashi.
\end{abstract}

\maketitle

\begin{footnotesize}
\tableofcontents
\end{footnotesize}

\section{Introduction}

In \hskip 1 pt the   paper \hskip 1 pt \cite{Kuznetsov}, Kuznetsov \hskip 1 pt discovered  his \hskip 1 pt trace formula for  $\PSL_2(\BZ)\backslash \BH^2 \cong \PSL_2(\BZ)\backslash \PSL_2(\BR)/K$, in which $\BH^2  $ denotes the hyperbolic upper half-plane and $K = \SO(2)/\{\pm 1\}$. There are two forms of his formula. The approach to the first formula is through a spectral decomposition formula for the inner product of {\it two} (spherical) Poincar\'e series. 
Then, using an inversion formula  for the Bessel transform, %\footnote{Along with the Lebedev-Kontorovich formula, it is named Bessel-Plancherel formula in \cite{Qi2} as a key ingredient in the study of Bessel functions for $\GL_3$. It is closely related to the Whitakker-Plancherel formula for $\SL_2 (\BR)$.}, 
Kuznetsov obtained another form of his trace formula (the version in \cite{D-I-Kuz} is more complete in the sense that holomorphic cusp forms occurring in the Petersson trace formula are also involved).
On the geometric side, a  weighted sum of Kloosterman sums arises from computing the Fourier coefficients of Poincar\'e series. The spectral side involves the Fourier coefficients of holomorphic and Maa\ss\ cusp forms and Eisenstein series along with the Bessel functions associated with their spectral parameters. The second form plays a primary role in the investigation of Kuznetsov on sums of Kloosterman sums in the direction of the Linnik-Selberg conjecture.\footnote{It is Selberg  who introduced Poincar\'e series and realized the intimate connections between Kloosterman sums and the spectral theory of the Laplacian on $ \PSL_2(\BZ)\backslash \BH^2 $ in \cite{Selberg}.}

Along the classical lines, the Kuznetsov trace formula has been studied and generalized by many authors (see, for example, \cite{Bruggeman-1,Bruggeman-2, Pro-Kuz,  D-I-Kuz, BM-Kuz-TR}). Their ideas of generalizing the formula to the non-spherical case are essentially the same as Kuznetsov. %to the first formula is through a spectral decomposition formula for the inner product of {\it two} Poincar\'e series. %It is effective to obtain information concerning the spectral parameters and the Fourier coefficients. The second formula is however more useful in the direction of the Linnik-Selberg conjecture on estimating the  averages of Kloosterman sums. 
It should however be noted that the pair of Poincar\'e series is chosen and spectrally decomposed in the space of a given $K$-type.

In the framework of representation theory, the second form of the Kuznetsov formula for an arbitrary Fuchsian group of the first kind $\Gamma \subset \PGL_2(\BR)$ was proved straightforwardly by Cogdell and Piatetski-Shapiro \cite{CPS}. Their computations use the Whittaker and Kirillov models of irreducible unitary representations of $\PGL_2(\BR) $. They observe that the Bessel functions occurring in the Kuznetsov  formula  should be identified with the Bessel functions for irreducible unitary representations of $\PGL_2(\BR) $ given by \cite[Theorem 4.1]{CPS}, in which the Bessel function $\EJ_{\pi} $ associated with a $\PGL_2 (\BR)$-representation $\pi$ satisfies the following kernel formula,
\begin{equation}\label{eq: Bessel function, Weyl element, real}
W \lp  
\begin{pmatrix}
a & \\
&  1
\end{pmatrix} 
\begin{pmatrix}
  & -1 \\
1 &  
\end{pmatrix}    \rp 
=  \int_{\BR^\times } \EJ_{\pi } ( a b) 
W 
\begin{pmatrix}
b & \\
& 1
\end{pmatrix} 
d^\times b,
\end{equation}
for all Whittaker functions $W$ in the Whittaker model of $\pi$.
Note that their idea of approaching  Bessel functions for $\GL_2(\BR)$  using local functional equations for $\GL_2 \times \GL_1$-Rankin-Selberg zeta integrals over $\BR$ (see \cite[\S 8]{CPS}) also occurs in \cite{Qi-Bessel}.

In  the book of  Cogdell and Piatetski-Shapiro \cite{CPS}, the Kuznetsov trace formula is derived  from computing the Whittaker functions (Fourier coefficients) of a ({\it single}) Poincar\'e series $P_f (g)$ in two different ways, first unfolding  $P_f (g)$ to obtain a weighted sum of Kloosterman sums, and  secondly spectrally expanding $P_f (g)$ in  $L^2 (\Gamma \backslash G)$ and then computing the Fourier coefficients of the spectral components in terms of basic representation theory of $\PGL_2 (\BR)$. The Poincar\'e series in \cite{CPS} arise from a very simple type of functions that are supported on the Bruhat open cell of $\PGL_2(\BR) $ and split in the Bruhat coordinates, which is surely not of a fixed $K$-type. In other words, \cite{CPS} suggests that, instead of the Iwasawa coordinates, it would be more pleasant to study the Kuznetsov   formula using the Bruhat coordinates. For this, \cite{CPS} works with the full spectral theorem rather than a version, used by all the other authors, that is restricted to a given $K$-type.

\vskip 5 pt

In the direction of generalization to other groups, Miatello and Wallach \cite{M-W-Kuz} gave the spherical Kuznetsov trace formula for real semisimple groups of real rank one, which include  both $\SL_2 (\BR)$ and $\SL_2 (\BC)$. It is however much more difficult to extend the formula to the non-spherical setting for $\SL_2 (\BC)$\footnote{This spherical formula was generalized to an arbitrary number field by Bruggeman and Miatello \cite{BM-Kuz-Spherical}.}. The first breakthrough is the work of Bruggeman and Motohashi \cite{B-Mo}, where the Kuznetsov trace formula for $\PSL_2( \BZ[i])\backslash \PSL_2( \BC)$ was found.\footnote{In her thesis \cite{B-Mo2}, Lokvenec-Guleska extended the  formula of \cite{B-Mo} to congruence subgroups of  $\SL_2 (\BC)$ over any imaginary quadratic field. Maga recently generalized the first form of the  formula to an arbitrary number field in \cite{Maga-Kuz}.} 

Let  %is the maximal compact subgroup of $\PSL(2, \BC)$ 
$\BH^3$ denote the three dimensional hyperbolic  space and let $K =  \SU(2)/ \{\pm 1\}$. %Their analysis is on the space $\BH^3 \times K$, which is isomorphic to $\PSL_2(  \BC)$ due to the Iwasawa decomposition. %The combination of the Jacquet and the Goodman-Wallach operators allows them to treat all the $K$-aspects. 
Featuring a combination of the Jacquet and the Goodman-Wallach operators, the analysis carried out in \cite{B-Mo} is considerably hard. Nevertheless, similar to \cite{Kuznetsov}, the approach of  \cite{B-Mo} is also from considering the inner product of {\it two}  certain sophisticatedly chosen Poincar\'e series of a given $K$-type.  It is however remarked without proof in \cite[\S 15]{B-Mo} that their Bessel kernel should be interpreted as the Bessel function of an irreducible unitary representation of $\PSL_2(  \BC)$. %This is justified by Proposition \ref{7prop: Bessel, C}. 
 
\vskip 10 pt

In this note, we shall prove the Kuznetsov trace formula for an {arbitrary} discrete group $\Gamma$ in $ \PGL_2 (\BC)$ that is cofinite but not cocompact. An essential ingredient is a kernel formula for $\PGL_2 (\BC)$ (see \eqref{eq: Bessel function, Weyl element}) that is identical to \eqref{eq: Bessel function, Weyl element, real}. 
This is a consequence of the representation theoretic investigations on a general type of special functions, the (fundamental) Bessel kernels for $\GL_n (\BC)$, arising in the \Voronoi summation formula; see \cite[\S 18]{Qi-Bessel}. 
Our derivation of the Kuznetsov formula will be in parallel with that in \cite{CPS}, at least for the Kloosterman-spectral formula (see \S \ref{sec: the Kloosterman-spectral formula}). As such, this note should be viewed as the supplement and generalization of the  work of Cogdell and Piatetski-Shapiro over the complex numbers. With this method, we can avoid the very difficult and complicated analysis in \cite{B-Mo}.

Finally, we remark that the only reason for considering $\PGL_2 (\BC)$ is to keep notations simple. Without much difficulty, we may extend the Kuznetsov formula to $  \SL_2 (\BC)$. %We just need to take into account of the representations of the type $\pi^{-}_d (i t)$. 
In another direction, we hope to implement the Cogdell-Piatetski-Shapiro method to $\PGL_2$ or $\SL_2$ over an arbitrary number field.

\begin{acknowledgement}
	The author would like to thank James W. Cogdell and Stephen D. Miller for valuable comments and helpful discussions. The author wishes to thank the anonymous referee for thorough reading of the manuscript and several suggestions which helped improving the paper.
\end{acknowledgement}

%\section{The Kuznetsov   formula for $ \PGL_2 (\BC)$}

\section{Notations and Statement of Theorem} 

\subsection{ } We shall adopt the notations in \cite{CPS}. Let $G = \PGL_2 (\BC)$($=\PSL_2 (\BC)$). Let
\begin{align*}
N & = \left\{ \underline u = \begin{pmatrix}
1 & u \\ & 1
\end{pmatrix} : u \in \BC \right\} \subset G, \\
A & =  \left\{ \underline a = \begin{pmatrix}
a &   \\ & 1
\end{pmatrix} : a \in \BCx \right\} \subset G.
\end{align*}
Define $A_+  =  \left\{ \underline r \in A : r \in \BR_+ \right\} $.  Let $B =  N A $ denote the Borel subgroup of $G$. Let $\varw = \begin{pmatrix}
 & - 1 \\ 1 & 
\end{pmatrix}$ be the long Weyl element of $G$. Then we have the Bruhat decomposition $G = B \cup N \varw B$. Let 
$$K = \SU (2)  / \{\pm 1 \} = \left\{ k_{v,\, w} = \begin{pmatrix}
v & w \\
- \overline w & \overline v
\end{pmatrix} : |v|^2 + |w|^2 = 1 \right\} /\{\pm 1 \}.$$ We have the Iwasawa decomposition $G = N A_+ K$.

Let $\mathfrak g$ denote the Lie algebra of $G$ and $\mathfrak U (\mathfrak g)$ its universal enveloping algebra.

\subsection{ }

Let $\Gamma \subset G$ be a discrete subgroup that is cofinite  but not cocompact. %If $P$ is any parabolic subgroup  of $G$, $P = M U$ its Langlands decomposition with $M$ the Levi and $U$ the unipotent radical, $P$ is called $\Gamma$-rational if $(\Gamma \cap U) \backslash U$ is compact.
Let $\mathcal C \subset \partial \BH^3$ denote the set of cusps of $\Gamma$. We assume that $ \infty \in \mathcal C$.  
For each cusp $\fra \in \mathcal C$, we fix $g_{\fra} \in G$ such that $ g_{\fra} \cdot \fra = \infty $. Let $P_{\fra}$ denote the parabolic subgroup of $G$ stabilizing $\fra$. Let $P_{\fra} = U_{\fra} M_{\fra}$ be its Langlands decomposition. Then the conjugation by $g_{\fra}$ provides an isomorphism $P_{\fra} \xrightarrow{\sim} B$, which induces isomorphisms $U_{\fra} \xrightarrow{\sim} N$ and $M_{\fra } \xrightarrow{\sim} A$. For each $\fra $,  define $\Gamma_{\fra}' = \Gamma \cap P_{\fra}$, $\Gamma_{\fra} = \Gamma \cap U_{\fra}$. Let $ g_{\fra} \Gamma_{\fra} g_{\fra}\-  = \left\{   \underline {u }  : u  \in \Lambda_{{\fra}} \right\} $, with $ \Lambda_{\fra}  $ a lattice in $\BC$. Let $|\Lambda_{\fra}| $ denote the area of $\Lambda_{\fra} \backslash \BC$. 
According to \cite[Theorem 2.1.8 (3)]{EGM}, we have $\Gamma_{\fra}' \cong \eta_{m_{\fra}} \ltimes \Lambda_{\fra}$, where $m_{\fra}  \in \left\{ 1, 2, 3, 4, 6 \right\}$, $\eta_{m_{\fra}}$ denotes the group of $m_{\fra}$-th roots of unity, and $\eta_{m_{\fra}} $ acts on $\Lambda_{\fra}$ by multiplication\footnote{ When $m_{\fra} =  4 $, respectively  $3$ or $6$, $\Lambda_{\fra}$ must be the ring of integers in the quadratic number field $\BQ (i)$, respectively $\BQ (\sqrt {-3})$.}. %$m_{\fra} = |\Gamma_{\fra}' / \Gamma_{\fra} |$ so that $\Gamma_{\fra}' \cong \eta_{m_{\fra}} \ltimes \Gamma_{{\fra}}$. 

%If we write an element of $G$ in the Iwasawa decomposition as $g = \underline u \cdot \underline b \cdot h $, 
%with $u \in \BC$, $b \in \BR_+ (\subset \BCx)$ and $h \in \SU (2)$, then we let $d g = d u \cdot b^{-2} d^\times b \cdot  d h $. Here $d h$ denote the measure on $\SU (2) \cong S^3 \subset \BR^4$ inherited from the Lebesgue measure on $\BR^4$.

%If we use the Bruhat coordinates $g = \underline {u}_1 \varw \underline {u}_2 \underline a$, with $u_1, u_2 \in \BC$ and $a \in \BCx$, then the measure is $d g = d u_1 d u_2 d^\times a$.

Let $\mathfrak X (\BCx)$ denote the group of multiplicative characters of $\BCx$. Every character $\mu \in \mathfrak X (\BCx)$ is of the form $\mu_{s, d} (z) =   | z |^{2 s} [z]^d$, with  $s \in \BC$, $d \in \BZ$ and the notation $ [z] = z/|z|$. Hence $\mathfrak X (\BCx)$   has a structure of a complex manifold, $\mathfrak X (\BCx) \cong \BZ \times \BC$. We define $\Re \mu_{s, d} = \Re (s)$. Each $\mu \in \mathfrak X (\BCx)$ defines a character of $A$, by $\mu (\underline a) = \mu (a)$, and a character of $B$ through $B \ra N \backslash B \cong A$. Through the isomorphisms give above, $ \mu$ also defines a character of $M_{\fra}$ and $P_{\fra}$. Let $\mathfrak X_m (\BCx)$ denote the set of characters $\mu_{s, d}$ such that $m | d$. Note that characters in  $\mathfrak X_{m_{\fra}} (\BCx)$ are trivial on $ \Gamma_{\fra}' $.

Let $\mathfrak X_{\infty} (\BC)$ denote the  group of additive characters  on $\Lambda_{\infty} \backslash \BC $. We define the dual lattice   $\Lambda_{\infty}'$ of $ \Lambda_{\infty} $ by $$\Lambda_{\infty}' = \left\{  \omega \in \BC :  \Tr (\omega \lambda) = \omega \lambda + \overline {\omega \lambda} \in \BZ,   \text{ for all } \lambda \in \Lambda_{\infty}    \right\}. $$%Let $\mathfrak X_{\infty}' (\BC)$ denote the set of orbits for the action of $\eta_{m_\infty}$ on $ \Lambda_{\infty} $.  %The orbit that contains $\psi$ will be denoted by $\{ \psi \} $. 
Then every $\psi \in \mathfrak X_{\infty} (\BC)$ is of the form $\psi_{\omega} (z) = e (\Tr (\omega z))$ for some $\omega \in \Lambda_{\infty}'$, with $e (x) = e^{2 \pi i x}$. Furthermore, each $\psi \in \mathfrak X_{\infty} (\BC)$ defines a character of $ \Gamma_{\infty} \backslash N $, by $\psi (\underline u) = \psi (u)$.

\subsection{ }

In the notations of \cite[\S 18.2]{Qi-Bessel}, the irreducible infinite dimensional unitary representations of $\PGL_2 (\BC)$ are
\begin{itemize}
	\item [-] (principal series) $\pi^+_d (i t)$ for $t \in \BR$ and $d \in \BZ$, with $ \pi^+_d (i t) \cong \pi^+_{- d} (- i t) $,
	\item [-] (complementary series) $  \pi (t)$ for $t \in \lp 0, \frac 12 \rp$.
\end{itemize}
In order to unify notations, let us put $\pi_d (i t) = \pi^+_d (i t)$ and $\pi_0 (t) = \pi (t)$ so that we may denote by $\pi_{d} (s)$ either   principal series for $s = it$ or complementary series for $s = t$ and $d = 0$.

\subsection{ }

Let $d z$ be {\it twice} the ordinary Lebesgue measure on $\BC$, and choose the standard multiplicative Haar measure $d^\times z = dz/|z|^2 $ on $\BCx = \BC \smallsetminus \{0\}$.

We take the Haar measure $dk$ on $K$ of total mass $1$. Writing an element of $G$ in the Iwasawa decomposition as $g = \underline z \hskip 1 pt \underline r \hskip 1 pt k $, 
with $\underline z   \in N$, $\underline r \in A_+ $ and $k \in K$, we let $d g = 2 r^{-3}  d z \hskip 1 pt d r  \hskip 1 pt  d k $. This Haar measure $d g$ on $G$ induces the hyperbolic measure on the hyperbolic space $\BH^3 = \left\{  z + r j : z \in \BC, r \in \BR_+ \right\}$. Moreover, if we use the Bruhat coordinates $g = \underline {u}_1  \varw \hskip 1 pt \underline {u}_2 \hskip 1 pt \underline a$ on the open Bruhat cell $N\varw B$, with $\underline u_1 , \underline u_2 \in N$ and $\underline a \in A$, then the measure $d g = (1/4 \pi^2)  d u_1 \hskip 1 pt d u_2 \hskip 1 pt d^\times a$. 

\subsection{Main Theorem}

\begin{thm}\label{thm: main}
	Let $F \in C^{\infty}_c (\BCx)$ be a smooth compactly supported function on $\BCx  $.   Let  $\omega_1, \omega_2 \in \Lambda_{\infty}' \smallsetminus \{0\}$ and $\psi_1 = \psi_{\omega_1}, \psi_2 = \psi_{\omega_2}$. Then
	\begin{equation*}
	\begin{split}
	\sum_{ c\in \Omega (\Gamma) } &  \frac {Kl_{\Gamma} (c; \psi_1, \psi_2 ) } {|c| } F \lp \frac {\omega_1 \omega_2} c \rp \\
	=  \ &
	  \pi    |\Lambda_{\infty}|^2  \sum_{ \sstyle \pi \in \Pi_d (\Gamma) \atop \sstyle \pi \cong \pi_d (s) } 
	\frac {  A_{\omega_2}(\varphi_{\pi}) \overline {A_{\omega_1 } (\varphi_{\pi})}   } 
	{  (2|d|+1)  G_{s, d}   }  \widehat F (s, d)  \\
	+ & \frac {1} { 4 }   \sum_{ \fra}  \frac 1 {m_{\fra} |\Lambda_{\fra}|} \sum_{ d\, \in\, m_{\fra} \BZ} \int_{\BR}  \mu_{it, d} \lp \frac {\omega_2} {\omega_1} \rp
	{ Z_{\Gamma}^{\fra} (\mu_{it, d} ; 0, \omega_2) \overline {Z_{\Gamma}^{\fra} (\mu_{it, d} ; 0, \omega_1) }   }   \widehat F (it , d)   \ d t,
	\end{split}
	\end{equation*}
	\delete{ where either $s = t$ is purely imaginary or $s = t$ is on the segment $\lp 0, \frac 1 2 \rp$,
		\begin{align*}
		G_{s, d} = 
		\begin{cases}
		1, & \text{ if } s = it, \\
		\ds \binom {2|d|} {|d|} \frac  { \Gamma (2  t + |d| + 1)^2} {  \Gamma (2|d| -2 t + 1)  \Gamma (2 t+1) }, & \text{ if } s = t.
		\end{cases}
		\end{align*}
		and
		\begin{equation*}%\label{7def: J mu m (z), n=2, C}
		J_{s, \pm 2 d} (z) = J_{- 2 s \mp d } \lp  z \rp J_{- 2 s \pm d  } \lp  {\overline z} \rp.
		\end{equation*}
	}
	where $Kl_{\Gamma} (c; \psi_1, \psi_2 )$ are the Kloosterman sums at infinity associated with $\Gamma$, $\psi_1$ and $\psi_2$ given in \S {\rm \ref{sec: Kloosterman sums}}, $A_{\omega} (\varphi_{\pi} )$ are the %{\rm(}normalized{\rm)} 
	Fourier coefficients of a certain $L^2$-normalized automorphic form in the space of the representation $\pi$ that occurs discretely in $L^2 (\Gamma \backslash G)$ {\rm(}see \S {\rm\ref{sec: compute the constants}}{\rm)}, $Z_{\Gamma}^{\fra} (\mu_{it, d} ; 0, \omega)$ are the Kloosterman-Selberg zeta functions  arising in the Fourier coefficients of Eisenstein series {\rm(}see \S {\rm\ref{sec: compute constants for Eisenstein series}}{\rm)}, 
	\begin{align*}
	G_{s, d} = 
	\begin{cases}
	1, & \text{ if } s = it, \\
	\ds    { \Gamma (1+2  t   ) } / {  \Gamma (1 -2 t  )   }, & \text{ if } s = t, d = 0,
	\end{cases}
	\end{align*}
	and finally $\widehat F (s, d)$ is the Bessel transform of $F$ given by
	\begin{equation*}
	\widehat F (s, d) = \frac 1 {\sin (2\pi s)} \int_{ \BCx} F (z)  \lp J_{ s, 2 d} (4 \pi   \sqrt { z}) -  J_{- s, - 2 d} (4 \pi   \sqrt { z}) \rp d^\times z,
	\end{equation*}
	with
	\begin{equation*}%\label{7def: J mu m (z), n=2, C}
	J_{s,   2 d} (z) = J_{- 2 s - d } \lp  z \rp J_{- 2 s + d  } \lp  {\overline z} \rp,
	\end{equation*}
	and $J_{\nu} (z)$   the classical Bessel function of the first kind.
\end{thm}

\begin{rem}
	When $\Gamma$ is a congruence group  for an imaginary quadratic field, the automorphic forms $\varphi_{\pi}$ will usually be chosen to be common eigenfunctions of Hecke operators. 
	Furthermore, according to \cite{Gelbart-Jacquet}, there will be no residual spectrum if $\Gamma$ is congruence. %, so the spectral parameter $s$ is always purely imaginary.  
	
	Note that our definition of Kloosterman sums is slightly different from the usual one. Taking the simplest example of the full modular group $ \Gamma = \PSL_2 (\mathfrak {O} ) $, with $\mathfrak {O} $ the ring of integers of an imaginary quadratic field, then for $c  \in \mathfrak {O} \smallsetminus \{0\} $ and $\omega_1, \omega_2 \in \mathfrak {O}'  \smallsetminus \{0\} $ we have
	\begin{align*}
	Kl_{\Gamma} (c^2; \psi_1, \psi_2) = \sum_{\sstyle a_1,\, a_2 \in \mathfrak {O}/ c  \mathfrak {O} \atop \sstyle a_1 a_2 \equiv 1 (\mod c ) } e \lp \Tr \lp \frac {\omega_1 a_1 + \omega_2 a_2}     c  \rp \rp.
	\end{align*}
	
	For $\Gamma = \PSL_2 (\BZ[i] )$ we obtain the summation formula in \cite{B-Mo}.
\end{rem}

\section{\texorpdfstring{Spectral Analysis of $L^2 (\Gamma \backslash G)$}{Spectral Analysis of $L^2 ( \Gamma  G)$}}
The spectral decomposition of $L^2 (\Gamma \backslash G)$ is a consequence of the general theory of Eisenstein series due to Langlands in \cite{Langlands-Eisenstein}. Here we shall follow the expositions in \cite{CPS}.

Let $L^2 (\Gamma \backslash G) $ be the space of all square integrable functions on $\Gamma \backslash G$ with respect to the measure induced by the Haar measure on $G$. Let $(\cdot, \cdot)$ be the (Petersson) inner product on $L^2 (\Gamma \backslash G)$. $G$ acts on this space by right translation.

Let $L^2_0 (\Gamma \backslash G)$ denote the space of $L^2$-cusp forms. We have a decomposition 
\begin{equation*}
L^2  (\Gamma \backslash G) = L^2_0 (\Gamma \backslash G) \oplus L^2_e (\Gamma \backslash G),
\end{equation*}
where $L^2_e (\Gamma \backslash G)$ is the orthogonal complement of  $ L^2_0 (\Gamma \backslash G)$.

\subsection{Spectral Decomposition on $L^2_0 (\Gamma \backslash G)$}

 It is a well known theorem of Gelfand and Piatetski-Shapiro  that $L^2_0 (\Gamma \backslash G)$ decomposes into a discrete countable direct sum of irreducible unitary representations $(\pi, V_\pi)$ of $G$, each isomorphism classes occurring with finite multiplicity (see  \cite{G-G-PS,Harish-Chandra,Langlands-Eisenstein}).
We let $\Pi_0(\Gamma)$ be the set of irreducible constituents of  $L^2_0 (\Gamma \backslash G)$.

\subsection{Eisenstein Series}

For each $\mu \in \mathfrak X(\BCx)$ and cusp $\fra \in \mathcal C$, let $V_{\fra} (\mu)$ denote the Hilbert space of functions $f : G \ra \BC$ such that
\begin{itemize}
	\item[-] $f (p g) = \mu (p) \delta_{\fra}^{\frac 1 2} (p) f (g), \hskip 10 pt p \in P_{\fra}$,
	\item[-] $f$ is square integrable on $K$.
\end{itemize}
Here $\delta_{\fra}$ is the modulus character of $P_{\fra}$ acting on $U_{\fra}$ by conjugation. For the matrix $p = g_{\fra}\- \underline {u} \hskip 1 pt \underline {a} g_{\fra}$ in $P_{\fra}$, we have $\delta_{\fra} (p) = |a|^2$.
$G$ acts on this space by right translation. We denote this representation by $\pi_{\fra} (\mu)$. There is a non-degenerate $G$-invariant Hermitian paring on $V_{\fra} (\mu) \times V_{\fra} (\overline \mu\-)$ given by
\begin{equation*}
\langle f , f' \rangle_{\fra } = \int_{P_{\fra} \backslash G} f  (g) \overline {f' (g)} d g, \hskip 10 pt f  \in V_{\fra} (\mu), \ f' \in V_{\fra} (\overline \mu\-).
\end{equation*}
Note that if $\Re (\mu) = 0$ then $\pi_{\fra} (\mu)$ is unitary with respect to this inner product.

For $f \in V_{\fra}$, we form the Eisenstein series
\begin{equation*}
E_{\fra } (g; f, \mu) = \sum_{\gamma \in \Gamma_{\fra} \backslash \Gamma } f (\gamma g).
\end{equation*}
Note that $E_{\fra } (g; f, \mu) \equiv 0$ if $\mu \notin \mathfrak X_{m_{\fra}}(\BCx)$. It converges uniformly and absolutely on compact subsets of $G$ for $\Re \mu > \frac 1 2$.  %(\cite[Proposition 3.1.3, 3.2.1]{EGM}).
$E_{\fra } (g; f, \mu) $ is analytic in $\mu$ and admits a meromorphic continuation to all $\mu \in \mathfrak X (\BCx)$.
Furthermore, there exist intertwining operators 
\begin{align*}
M (\fra, \frb; \mu) : V_{\fra} (\mu) \ra V_{\frb} (\mu\-)
\end{align*}
such that the Eisenstein series satisfy the functional equation \begin{equation}\label{eq: functional equation for Eisenstein}
E_{\fra} (g; f,\mu) = \sum_{\frb} E_{\frb} (g; M (\fra, \frb; \mu) f, \mu\-).
\end{equation}

\subsection{The Residual Spectrum} 
The poles of the $E_{\fra} (g; f,\mu)$ in $\Re \mu \geq 0$, which are identical with the poles of $M (\fra, \frb; \mu)$, are all simple and their $s$-coordinates lie in a finite subset of the segment $ \left(0,\frac 1 2\right]$. The residues of these Eisenstein series at such a pole form a non-cuspidal irreducible representation occurring discretely in $L^2 (\Gamma \backslash G)$. The point $\mu = \delta_{\fra}^{\frac 1 2}$ always yields the trivial representation, which represents the constant functions, whereas all  the other components in the residual spectrum are complementary series. Let $\Pi_r (\Gamma)$ denote the representations $(\pi, V_\pi)$ which so occur discretely in $L^2 (\Gamma \backslash G)$ coming from the residues of all Eisenstein series $E_{\fra} (g; f, \mu)$. 

\subsection{The Continuous Spectrum} It is known that $L^2_e (\Gamma \backslash G)$ is the closure of the space spanned by incomplete Eisenstein series\footnote{Here we follow the terminologies and notations of \cite{Iw-Spectral} rather than \cite{CPS}. }  $E_{\fra} (g; f)$ of the form
\begin{equation*}
E_{\fra} (g; f) = \sum_{\Gamma_\fra \backslash \Gamma} f (\gamma g),
\end{equation*}
with $f \in C_c^\infty (U_{\fra} \backslash G)$. For $f \in C_c^\infty (U_{\fra} \backslash G)$, $E_{\fra} (g; f)$ is compactly supported on $\Gamma \backslash G$, then we may form the inner product
\begin{equation*}
(E_{\fra} (\cdot; f), E_{\frb} (\cdot; f', \mu) )
\end{equation*}
with the Eisenstein series for $f' \in V_{\frb} (\mu)$ with $\Re \mu = 0$. This linear functional will be presented by some   $\Phi_{\pi_{\frb} (\mu)} (E_{\fra} (\cdot; f)) \in V_{\frb} (\mu)$ in the sense that
\begin{equation*}%\label{eq: def Phi Eisenstein}
(E_{\fra} (\cdot; f), E_{\frb} (\cdot; f', \mu) ) = \langle \Phi_{\pi_{\frb} (\mu)} (E_{\fra} (\cdot ; f)), f' \rangle_{\frb}, \hskip 10 pt \text{   all } f' \in V_{\frb} (\mu).
\end{equation*}
Therefore $\Phi_{\pi_{\frb} (\mu) }$ gives a $G$-intertwining projection  $L_e^2 (\Gamma\backslash G) \ra V_{\frb} (\mu)$.

\subsection{The Spectral Decomposition of $L^2 (\Gamma \backslash G)$}

For convenience, let $\Pi_d (\Gamma)  = \Pi_0 (\Gamma) \cup \Pi_r (\Gamma)$ denote the complete discrete spectrum of $L^2 (\Gamma \backslash G)$.

\begin{thm} 
	The spectral decomposition of $L^2 (\Gamma \backslash G)$ is 
	\begin{equation*}
	L^2 (\Gamma \backslash G) = \lp \bigoplus _{\pi \in \Pi_d (\Gamma) } V_{\pi} \rp \oplus \sum_{\fra} \frac 1 {4 \pi i} \frac 1 {m_{\fra} |\Lambda_{\fra}|} \int_{ \Re \mu = 0} V_{\fra} (\mu) d \mu,
	\end{equation*}
	where  the sections of the continuous spectrum must satisfy the functional equation \eqref{eq: functional equation for Eisenstein}  and the integral over $d \mu$ is on the vertical lines $\mathfrak X_{m_{\fra}} (\BCx) \cong m_{\fra} \BZ \times \{ s \in \BC : \Re s = 0 \}$. %Note that those $\mu \notin \mathfrak X_{m_{\fra}} (\BCx)$ have zero contribution to the continuous spectrum. 
	
	To be concrete, for a function $\varphi \in L^2 (\Gamma \backslash G) $,  we may  write
	\begin{equation}\label{eq: spectral decomposition of varphi}
	\varphi = \sum_{ \pi \in \Pi_d (\Gamma) } F_{\pi} (\varphi) + \sum_{ \fra} \frac 1 {4 \pi i} \frac 1 {m_{\fra} |\Lambda_{\fra}|} \int_{\Re \mu = 0}  F_{\pi_{\fra} (\mu) } (\varphi) d \mu.
	\end{equation}
	For $\pi \in \Pi_d (\Gamma)$, $F_{\pi} (\varphi)$ is the unique element in $V_{\pi}$ satisfying 
	\begin{equation}\label{eq: def of F (phi)}
	(F_{\pi} (\varphi), \varphi' ) = (\varphi, \varphi'), \hskip 10 pt \text{   all } \varphi' \in V_{\pi}.
	\end{equation}
	Let $\Phi_{\pi_{\fra} (\mu)}$ be the $G$-intertwining projection from $L^2_e (\Gamma \backslash G) \ra V_{\fra} (\mu)$, which extends onto $L^2 (\Gamma \backslash G) $, then
	\begin{equation}\label{eq: def F Eisenstein}
	F_{\pi_{\fra} (\mu)} (\varphi) = E_{\fra} (g; \Phi_{\pi_{\fra} (\mu)} (\varphi), \mu ).
	\end{equation}
\end{thm}

The spectral decomposition theorem in general is   due to Langlands (\cite{Langlands-Eisenstein}). For $ \PGL_2 (\BC)$, we may however prove the theorem following \cite[Chapter 6]{EGM},   \cite{Gelbart-Jacquet} or \cite{Iw-Spectral} for $\PSL_2 (\BR)$. A Fourier series expansion on the circle $  B \cap \SU(2)/ \{\pm 1 \} \cong \SU (1) / \{ \pm 1 \} $ is needed at the end.

Let $S (\Gamma \backslash G)$ be the space of smooth vectors in $L^2 (\Gamma\backslash G)$, which is endowed with a finer topology than that induced by the $L^2$-norm.  This topology may be defined by the seminorms $\nu_X (\varphi) = \| R(X) \varphi\|_2 $ for all $X \in  \mathfrak U (\mathfrak g) $. Recall that $\mathfrak g$ is the Lie algebra of $G$ and $ \mathfrak U (\mathfrak g)$ is the universal enveloping algebra. Let $V^{\infty}_{\pi}$ and $V_{\fra} (\mu)^{\infty}$ denote the spaces of smooth vectors in $V_{\pi} $ and $V_{\fra} (\mu)$, respectively. As consequences of the theorem of Dixmier-Malliavin \cite{Dix-Mal}, for $\varphi \in S (\Gamma \backslash G) $, each of its spectral components is smooth and the spectral decomposition \eqref{eq: spectral decomposition of varphi} converges in $S (\Gamma \backslash G)$ (see \cite[Proposition 1.3, 1.4]{CPS}). %More precisely, if we let $V^{\infty}_{\pi}$ and $V_{\fra} (\mu)^{\infty}$ denote the spaces of smooth vectors in $V_{\pi} $ and $V_{\fra} (\mu)$, respectively,  then $F_{\pi} (\varphi) \in V_{\pi}^{\infty}$ and $ \Phi_{\pi_{\fra} (\mu)} (\varphi) \in V_{\fra} (\mu)^{\infty} $ when $\varphi \in S (\Gamma \backslash G)$.

\subsection{The Whittaker-Spectral Decomposition}

	Let $\psi \in \mathfrak X_{\infty} (\BCx) \smallsetminus \{ 1 \}$ and  $\varphi \in S (\Gamma \backslash G)$. 
	We define the $\psi$-Whittaker function associated with $\varphi $ as
	\begin{equation*}
	W_{\varphi,\, \psi} (g) = \int_{ \Gamma_{\infty} \backslash N} \varphi (n g) \psi\- (n) d n.
	\end{equation*}

\delete{	 {\rm (2).} For $\chiup \in \mathfrak X (\BCx)$, we define the Mellin-Whittaker transform of $\varphi$ as
	\begin{equation*}
	M_{\psi} (\varphi,\, \chiup) = \int_{\BCx} W_{\varphi,\, \psi} \lp \begin{pmatrix}
	a & \\ & 1
	\end{pmatrix} \rp \chiup (a) d^\times a
	\end{equation*}
	when the integral converges.}

Let $ C (N\backslash G; \psi)$ denote the space of continuous functions $f$ on $G$ such that $f (n g) = \psi (n) f (g)$ for all $n \in N$. Let $ S (N\backslash G; \psi)$ denote the space of smooth vectors in $ C (N\backslash G; \psi)$. It is clear that if $\varphi \in S (\Gamma \backslash G)$ then $W_{\varphi,\, \psi} \in S (N\backslash G; \psi)$.

Using Sobolev's lemma (\cite[\S 5.6.3 Theorem 6]{Evans-PDE}), we may prove the following analogue of \cite[Lemma 1.1]{CPS}. For all $\varphi \in S (\Gamma \backslash G) $ we have
\begin{equation*}
\left| W_{\varphi,\, \psi} (\underline {a} ) \right| \lll (|a|^4 + |a|^{-4} ) \sum_{\sstyle X = \boldsymbol{X}^{\boldsymbol{\alpha} } \atop \sstyle |\boldsymbol{\alpha} | \leqslant 4} \| R(X) \varphi \|_2, 
\end{equation*}
where, on choosing a basis $\left\{ X_{l} \right\}_{l=1}^6 $ of $\mathfrak g$, $\boldsymbol{X}^{\boldsymbol{\alpha}} = \prod_{l = 1}^6 X_l^{\alpha_l} \in \mathfrak U (\mathfrak g)$ (the order of $X_l$ in the product is fixed), $|\boldsymbol{\alpha} | = \sum_{l=1}^6 \alpha_l  $, and the implied constant depends only on $\Gamma$.
Consequently, the linear functional $\varphi \mapsto W_{\varphi,\, \psi} (1)$ is continuous on $S (\Gamma \backslash G)$ with its natural topology.

\begin{thm}\label{thm: Whittaker}
	Let $\psi \in \mathfrak X_{\infty} (\BCx) \smallsetminus \{ 1 \}$. Suppose $\varphi \in S  (\Gamma \backslash G)$ has spectral expansion as in \eqref{eq: spectral decomposition of varphi}.
	%{\rm (1).} 
	Then
	\begin{equation}\label{eq: Whittaker spectral decomposition}
	\displaystyle W_{\varphi,\, \psi} (1) =  \sum_{ \pi \in \Pi_d (\Gamma) } W_{ F_{\pi} (\varphi),\, \psi } (1) + \sum_{ \fra} \frac 1 {4 \pi i} \frac 1 {m_{\fra} |\Lambda_{\fra}|} \int_{\Re \mu = 0}  W_{F_{\pi_{\fra} (\mu) } (\varphi) ,\, \psi} (1) d \mu,
	\end{equation}
	with $\mu \in \mathfrak X_{m_{\fra}} (\BCx)$.
	
	\delete{{\rm (2).} For $\Re \chiup > 2$, we have
	$$ M_{\psi}  (\varphi, \chiup) =  \sum_{ \pi \in \Pi_d (\Gamma) } M_{\psi}  (  F_{\pi} (\varphi), \chiup) + \sum_{ \fra} \frac 1 {4 \pi i} \int_{\Re \mu = 0}  M_{\psi}  ( F_{\pi_{\fra} (\mu) } (\varphi), \chiup) d \mu.$$}
\end{thm}

%There are some subtleties here with the continuous spectrum, involving the extending the definition of the  Whittaker integrals onto Eisenstein series and interchanging the order of integrations.

\section{Kloosterman Sums, Poincar\'e Series and the Kloosterman-Spectral Formula}\label{sec: the Kloosterman-spectral formula}

In the following, we fix two nontrivial characters $\psi_1, \psi_2 \in \mathfrak X_{\infty} (\BC)  $. Let $\kappa \in \BCx$ be such that $\psi_2 (z) = \psi_1 (\kappa z)$.

\subsection{Kloosterman Sums} \label{sec: Kloosterman sums}

We first introduce
\begin{equation*}
\Omega (\Gamma) = \left\{ c \in \BCx : N \varw \underline {c} N \cap \Gamma \neq \O, \, \underline c \in A  \right\},
\end{equation*}
and for $c \in \Omega (\Gamma) $ define $\Gamma_c = N \varw \underline {c} N \cap \Gamma $. $\Gamma_c$ is both right and left invariant under $\Gamma_{\infty}$. For each $\gamma\in \Gamma_c$, we decompose $\gamma$ according to the Bruhat decomposition, namely
\begin{equation*}
\gamma = n_1 (\gamma) \varw \underline {c} n_2 (\gamma),
\end{equation*}
with $n_1 (\gamma), n_2 (\gamma) \in N$.

	For $c \in \Omega (\Gamma)$ and $\psi_1, \psi_2 \in \mathfrak X_{\infty} (\BCx) $, we define the associated Kloosterman sum as 
	\begin{equation*}%\label{eq: Kloosterman sum 1}
	Kl_{\Gamma} (c; \psi_1, \psi_2) = \sum_{ \Gamma_{\infty} \backslash \Gamma_c / \Gamma_{\infty} } \psi_1 (n_1 (\gamma)) \psi_2 (n_2 (\gamma)).
	\end{equation*}

%By Shimizu's Lemma \cite[Theorem 1.3.1]{EGM}, there exists a positive constant $c_{\Gamma}$ such that for all $c \in \Omega (\Gamma)$ we have $|c| > c_{\Gamma}$.

\delete{ 
%\subsection{  }
%We may  define an intertwining operator
%\begin{equation*}
%K (c; \psi_1, \psi_2) :  S (N\backslash G; \psi_1) \ra  S (N\backslash G; \psi_2),
%\end{equation*}
%given by 
For $f \in S (N\backslash G; \psi_1) $, if there is some constant $\alpha > 0$ such that 
\begin{equation*}
| f (\underline z \hskip 1 pt \underline r k) | \lll_{\alpha }   r^{1+\alpha} 
\end{equation*}
for all $z \in \BC$, $r \in \BR_+$ and $k \in K$, then the integral
\begin{equation*}
K (c; \psi_1, \psi_2) (f) (g) = \int_N f \lp \varw \underline c n g \rp \psi_2\- (n) d n
\end{equation*}
converges absolutely and defines a function in $  S (N\backslash G; \psi_2)$. This is the intertwining operator of Jacquet.
%Furthermore, note that
%$K(f, c; \psi_1, \psi_2) = K (c; \psi_1, \psi_2) (f) (1) $ gives rise to a distribution on $ S (N\backslash G; \psi_1) $.
}

\subsection{Poincar\'e Series}

Given any $\alpha, \beta > 0$, we let $S_{\alpha,\, \beta} (N \backslash G; \psi_1) $ denote the function space consisting of functions $f \in S (N \backslash G; \psi_1)$ satisfying
\begin{equation*}
| f (\underline z \hskip 1 pt \underline r k) | \lll_{\alpha,\, \beta} \min \left\{ r^{1+\alpha}, r^{1-\beta} \right\}
\end{equation*}
for all $\underline z \in N$, $\underline r \in A_+$ and $k \in K$. %We say that $f$ is a Schwartz function modulo $N$ if $R (X) f$ is rapidly decreasing modulo $N$  for all $X \in \mathscr  U (\mathfrak g) $.
%Let $\mathscr S (N \backslash G; \psi_1)$ denote the subspace of $S (N \backslash G; \psi_1)$ consisting of the functions that are Schwartz functions modulo $N$. 
For $f \in S_{\alpha,\, \beta} (N \backslash G; \psi_1)$, we form the Poincar\'e series
\begin{equation*}
P_f (g) = \sum_{ \Gamma_{\infty} \backslash \Gamma} f (\gamma g).
\end{equation*}
Using  the (spherical) Eisenstein series in \cite[\S 3.2]{EGM} as   majorant, it is readily verified that the series for $P_f (g)$ is absolutely convergent, uniformly on compact subsets. Moreover, applying similar arguments as in the proof of \cite[Proposition 3.2.3]{EGM}, we may estimate $P_f (g)$ near the cusps of $\Gamma$ and prove that $P_f (g) \in S (\Gamma \backslash G)$.\footnote{In \cite[\S 2.3]{CPS}, the notion of rapid decreasing modulo $N$ is introduced in terms of the norm $\|\, \|_N$ on $N \backslash G$. However, their arguments in the proofs of Proposition 2.1 and 2.2 are incorrect. Nevertheless, the arguments here work in the $\PGL_2(\BR)$ context and corrects the errors in \cite{CPS}.} 

Later, we shall apply the Whittaker-spectral formula \eqref{eq: Whittaker spectral decomposition} in Theorem \ref{thm: Whittaker} to a specific Poincar\'e series $P_f$ that will be constructed in a moment. With this in mind, we need the following results that follow from   standard unfolding computations. 

By decomposing $\Gamma$ according to the Bruhat decomposition,
\begin{equation*}
\Gamma = \Gamma_{\infty}' \cup \bigcup_{c \in \Omega(\Gamma)} \Gamma_c, 
\end{equation*}
we find  on the left hand side of \eqref{eq: Whittaker spectral decomposition} a weighted sum of Kloosterman sums.
\begin{lem} \label{lem: geometric side}
	Let $ f \in S_{\alpha,\, \beta} (N \backslash G; \psi_1) $. Then
	\begin{equation*}%\label{eq: geometric side}
	W_{P_f,\, \psi_2} (1) = \delta_{\psi_1,\, \psi_2} |\Lambda_{\infty}| f (1)   + \sum_{c \in \Omega (\Gamma) }
	Kl_{\Gamma} (c; \psi_1, \psi_2) K(f, c; \psi_1, \psi_2),
	\end{equation*}
	where $\delta_{\psi_1,\, \psi_2}  $ is the Kronecker symbol for $\psi_1 $ and $\psi_2$ lying in the same orbit under the action of $\eta_{m_\infty}$, $ |\Lambda_{\infty}| $ is the area of $\Lambda_{\infty} \backslash \BC$, and 
	\begin{equation*}%\label{eq: Kloosterman distribution}
	K (f, c; \psi_1, \psi_2)    = \int_N f \lp \varw \underline c n \rp \psi_2\- (n) d n
	\end{equation*}
	is called the Kloosterman-Jacquet distribution.
\end{lem}

For the right hand side of \eqref{eq: Whittaker spectral decomposition},  we shall need two identities on the inner products with Poincar\'e series. First,
\begin{equation}\label{eq: inner product identity}
(P_f, \varphi) = \int_{ N\backslash G} f (g) \overline {W_{\varphi,\, \psi_1} (g) } d g, \hskip 10 pt \varphi \in S (\Gamma\backslash G).
\end{equation}
Second, for $f \in S_{\alpha,\, \beta} (N \backslash G; \psi_1)$, as alluded to above, with the arguments in the proof of \cite[Proposition 3.2.3]{EGM}, one may show that the Poincar\'e series $P_f (g)$ has sufficient decay to take its inner product with the Eisenstein series $E_{\fra} (g ; f', \mu)$, and 
\begin{equation}\label{eq: inner product identity, Eisenstein}
(P_f, E_{\fra} (\cdot ; f', \mu) ) = \int_{ N\backslash G} f (g) \overline {W_{E_{\fra} (\cdot ; f',\, \mu),\, \psi_1} (g) } d g, \hskip 10 pt f' \in V_{\fra} (\mu).
\end{equation}

\subsection{The Specific Choice of Poincar\'e Series}\label{sec: choice of f}

Let $\eta   \in \mathscr S (\BC)$ be a Schwartz function on $\BC$ and let $\nu \in C^{\infty}_c (\BCx)$ be a smooth function of compact support on $\BCx$. To $\eta$ and $\nu$ we shall associate a function $f_{\eta, \nu} \in S (N \backslash G; \psi_1)$ by defining
\begin{equation*}
f_{\eta,\nu} (g) = 
\begin{cases}
\psi_1 (u_1) \eta (u_2) \nu (a), & \text{ if } g = \underline {u}_1 \varw \underline u_2 \hskip 1 pt \underline a \in N \varw N A, \\
0, & \text{ if } g \in B = N A.
\end{cases}
\end{equation*}
%Since $\nu$ is compactly supported, $f_{\eta, \nu}$ vanishes in a neighborhood of the small Bruhat cell $B$. 
We have $f_{\eta,\nu} (n g) = \psi_1 (n) f_{\eta,\nu} (g) $ for $n \in N$ and it may be shown that $f_{\eta,\nu} \in S_{\alpha, 1} (N\backslash G; \psi_1) $ for all $\alpha > 0$. Note that $f_{\eta, \nu} (1) = 0$.

Subsequently, for the specific choice $\varphi = P_{f_{\eta,\nu}}$, the left and the right hand side of the identity \eqref{eq: Whittaker spectral decomposition} will be referred to as the geometric and the spectral side, respectively.

\subsection{The Geometric Side}

The Kloosterman-Jacquet distribution $K(f, c; \psi_1, \psi_2)$ in Lemma \ref{lem: geometric side} for $ f = f_{\eta, \nu} $ can be computed very easily.

Since
\begin{equation*}
f_{\eta, \nu} \lp \varw \begin{pmatrix}
c & \\
 & 1
\end{pmatrix} 
\begin{pmatrix}
1 & u \\
  & 1
\end{pmatrix}\rp = f_{\eta, \nu} \lp \varw 
\begin{pmatrix}
1 & c u \\
& 1
\end{pmatrix}
\begin{pmatrix}
c & \\
& 1
\end{pmatrix} \rp
= \eta (c u) \nu (c),
\end{equation*}
by definition, %\eqref{eq: Kloosterman distribution}, 
we have
\begin{equation*}
K (f_{\eta, \nu}, c; \psi_1, \psi_2) = \int_{\BC} \eta (c u) \nu (c) \psi_2\- (u) d u = \frac {1} {|c|^2 }  \nu (c) \int_{\BC} \eta ( u)  \psi_2\- \lp \frac {u} c \rp d u.
\end{equation*}
\begin{lem}
We have
\begin{equation*}
K (f_{\eta, \nu}, c; \psi_1, \psi_2) = \frac 1 {|c|^2} \nu (c) \widehat \eta \lp \frac 1 c \rp,
\end{equation*}
where $\widehat \eta $ is the Fourier transform with respect to $\psi_2$,
$$\widehat \eta (z) = \int_{\BC} \eta (u) \psi_2\- (u z) d u. $$
\end{lem}

\subsection{The Spectral Side}

In the following, we shall compute the spectral side  and show that it can be expressed in terms of Bessel transforms of the weight function.
\subsubsection{Bessel functions}
Let $\psi$ be a nontrivial additive character of $\BC$. For an infinite dimensional  unitary representation $\pi$ of $G$, let $\mathcal W (\pi, \psi )$ be the Whittaker model of $\pi$, that is, the image of a nonzero embedding of $V_{\pi}^{\infty}$ into the space $\mathrm{Ind}_N^G (\psi)$ of $C^{\infty}$ functions $W : G \ra \BC$ satisfying $W (n g) = \psi (n) W(g)$ for all $n \in N$, with $G$ acting on $\mathrm{Ind}_N^G (\psi)$ by right translation. %We denote by $W_{v}$ the image of $v \in V_{\pi}^{\infty}$ under the embedding. 
Actually, it is known from Shalika's multiplicity one theorem \cite{Shalika-MO} that such a nonzero embedding exists and is unique up to constant. Let $\mathcal K (\pi, \psi )$ be the associated Kirillov model of $\pi$   on $ L^2 (\BCx, d^{\times} z) $, whose smooth vectors are all of the form $  W  \begin{pmatrix}
a & \\ & 1
\end{pmatrix}  $ for $W  \in \mathcal W (\pi, \psi )$. %The Kirillov model comes equipped with its canonical $L^2$ inner product on $\BCx$.

It is shown in \cite[(18.4)]{Qi-Bessel} that there exists a (complex-valued) real analytic function $\EJ_{\pi, \, \psi }$ on $\BCx$ which acts as an integral kernel of the action of the Weyl element $\varw$ on the Kirillov model $\mathcal K (\pi, \psi)$. Namely, for $W \in \mathcal W (\pi, \psi)$,
\begin{equation}\label{eq: Bessel function, Weyl element}
W \lp  
\begin{pmatrix}
a & \\
&  1
\end{pmatrix} 
\varw   \rp 
=  \int_{\BC^\times } \EJ_{\pi, \, \psi } ( a b) 
W 
\begin{pmatrix}
b & \\
& 1
\end{pmatrix} 
d^\times b.
\end{equation}
A simple fact is that the Bessel function is conjugation invariant, namely, $\overline { \EJ_{\pi,\, \psi } (z) } = \EJ_{\pi, \, \psi\- } (z) = \EJ_{\pi , \, \psi } (z)$ (see \cite[(18.2)]{Qi-Bessel}).

\subsubsection{The constants $c_{\psi } (\pi, \Gamma) $ and $c_{\psi } (\pi_{\fra} (\mu), \Gamma) $} 

Let $\psi$ be a nontrivial character of $\BC$. 

First, let $(\pi, V_{\pi})$ be a discrete component of $L^2 (\Gamma \backslash G)$. The inner product on $V_{\pi}$ is the Petersson inner product inherited from $L^2 (\Gamma \backslash G)$. For  smooth vector $\varphi \in V_{\pi}^\infty \subset S (\Gamma\backslash G)$, the map
\begin{equation*}
\varphi \mapsto W_{\varphi} (g) = \int_{\Gamma_{\infty}\backslash N} \varphi (n g) \psi \- (n) d n
\end{equation*}
defines a Whittaker model $\mathcal W (\pi, \psi )$ of $(\pi, V_{\pi}^{\infty})$. Let $\mathcal K (\pi, \psi )$ be the associated Kirillov model. %So  $w_{\varphi} (a) = W_{\varphi} \begin{pmatrix}
%a & \\ & 1
%\end{pmatrix} $ for $\varphi \in V_{\pi}^{\infty}$. 
It comes equipped with its canonical $L^2$ inner product on $\BCx$. Hence there is a positive constant, which we shall denote $c  (\pi, \Gamma) = c_{\psi } (\pi, \Gamma)$, such that for all $\varphi, \varphi' \in V_{\pi}$
\begin{equation}\label{eq: first constant}
\int_{\BCx} W_{\varphi} \begin{pmatrix}
a & \\ & 1
\end{pmatrix} \overline {W_{\varphi'} } \begin{pmatrix}
a & \\ & 1
\end{pmatrix} d^\times a = c  (\pi, \Gamma) \int_{ \Gamma \backslash G} \varphi (g) \overline {\varphi' (g)} d g.
\end{equation}

Second, we consider the representation $(\pi_{\fra} (\mu), V_{\fra} (\mu) )$. For $f \in V_{\fra} (\mu)$, the function
\begin{align*}
W_{E_{\fra} (\cdot; f ,\, \mu)  } (g) = \int_{\Gamma_{\infty}\backslash N} E_{\fra} (ng; f, \mu) \psi \- (n) d n
\end{align*}
defines a Whittaker model $\mathcal W (\pi_{\fra} (\mu) , \psi )$
of $\pi_{\fra} (\mu)$. %Note that this Whittaker model is defined via intertwining $(\pi_{\fra} (\mu), V_{\fra} (\mu) )$ into the automorphic forms on $\Gamma \backslash G$.
%We let $\mathcal K (\pi_{\fra}, \psi )$ denote the associated Kirillov model. 
Recall that there is also a canonical inner product on $(\pi_{\fra} (\mu), V_{\fra} (\mu) )$ given by
\begin{equation*}
\langle f, f' \rangle_{\fra} = \int_{ P_{\fra} \backslash G} f(g) \overline {f' (g)} d g.
\end{equation*}
Hence  there exists a  constant $c  (\pi_{\fra} (\mu), \Gamma) = c_{\psi } (\pi_{\fra} (\mu), \Gamma) $
such that for all $f, f' \in V_{\fra} (\mu)^{\infty}$ we have
\begin{equation}\label{eq: first constant, Eisenstein}
\int_{\BCx} W_{E_{\fra} (\cdot; f ,\, \mu)  } \begin{pmatrix}
a & \\ & 1
\end{pmatrix} \overline {W_{E_{\fra} (\cdot; f' ,\, \mu)  } } \begin{pmatrix}
	a & \\ & 1
	\end{pmatrix}  d^\times a
= c (\pi_{\fra} (\mu), \Gamma) \int_{ P_{\fra} \backslash G} f (g) \overline {f' (g)} d g.
\end{equation}

\subsubsection{The constants $ c (\pi ; \psi_2/\psi_1) $ and $ c (\pi_{\fra} (\mu); \psi_2/\psi_1) $}

In what follows we shall have to compare Whittaker models associated with $\psi_1$ and $\psi_2$. %For this, it is convenient to introduce $\kappa = \omega_2 / \omega_1$ so that $\psi_2 (z) = \psi_1 (\kappa z)$.

Let $\pi \in \Pi_d (\Gamma)$. Because $\psi_2 (z) = \psi_1 (\kappa z)$, if $W_{\varphi,\, \psi_1} \in \mathcal W (\pi, \psi_1) $, then the function
\begin{equation*}
W'_{\varphi,\, \psi_2} (g) = W_{\varphi,\, \psi_1} \lp \begin{pmatrix}
\kappa & \\
 & 1
\end{pmatrix} g \rp
\end{equation*}
satisfies $ W'_{\varphi,\, \psi_2} (n g) = \psi_2 (n) W'_{\varphi,\, \psi_2} (  g) $. Since this left multiplication by $ \begin{pmatrix}
\kappa & \\
& 1
\end{pmatrix} $ commutes with the action of $G$, the collection of function $ W'_{\varphi,\, \psi_2} (  g) $ form another $\psi_2$-Whittaker model. By Shalika's multiplicity one theorem, there is a constant $c (\pi; \psi_2/\psi_1) $ such that for all $ \varphi \in V_{\pi}^{\infty}$ we have
\begin{equation}\label{eq: second constant}
W_{\varphi,\, \psi_2} (  g) = c (\pi; \psi_2/\psi_1) W_{\varphi,\, \psi_1} \lp \begin{pmatrix}
\kappa & \\
& 1
\end{pmatrix} g \rp.
\end{equation}

Similarly, there exists a constant $ c (\pi_{\fra} (\mu); \psi_2/\psi_1) $ such that for all $f \in V_{\fra} (\mu)^{\infty}$ we have
\begin{equation}\label{eq: second constant, Eisenstein}
W_{E_{\fra} (\cdot; f ,\, \mu),\, \psi_2} (  g) = c (\pi_{\fra} (\mu); \psi_2/\psi_1) W_{E_{\fra} (\cdot; f ,\, \mu),\, \psi_1} \lp \begin{pmatrix}
\kappa & \\
& 1
\end{pmatrix} g \rp.
\end{equation}

\delete{
%\subsection{The Kloosterman-spectral formula}

For $\pi$ occurring in either the discrete or continuous spectrum of $L^2 (\Gamma \backslash G)$ we let
\begin{equation*}
c_{\Gamma} (\pi; \psi_1, \psi_2) = \frac {c(\pi; \psi_2/\psi_1) c_{\psi_1}( \pi, \Gamma ) } {4 \pi^2 }.
\end{equation*}
Furthermore, for $\eta \in \mathscr S (\BC)$ and $\nu \in C^\infty_c (\BCx)$ we set 
$$A_{\eta, \nu} (z) = \nu \lp \frac 1 z \rp \widehat \eta (z), \hskip 10 pt z \in \BCx,$$
where $ \widehat \eta $ is the Fourier transform of $\eta$ with respect to $\psi_2$.}

\subsubsection{Computing the spectral side}

\begin{lem}
	Let $f = f_{\eta, \nu}$. Suppose that $\psi_2 (z) = \psi_1 (\kappa z)$. 
	
	For $(\pi, V_{\pi})$ occurring discretely in $L^2 (\Gamma \backslash G)$, we have
	\begin{equation}\label{eq: spectral side, discrete}
	W_{F_{\pi} (P_f),\, \psi_2 } (1) = c_{\Gamma} (\pi; \psi_1, \psi_2) 
	\int_{\BCx} \nu (z)  \widehat \eta \lp \frac 1 z \rp \EJ_{\pi, \, \psi_1} \lp \frac {\kappa} z \rp d^\times z,
	\end{equation}
	with
	\begin{equation*}
	c_{\Gamma} (\pi; \psi_1, \psi_2) = \frac {c(\pi; \psi_2/\psi_1) c_{\psi_1}( \pi, \Gamma ) } {4 \pi^2 }.
	\end{equation*}
	
	For $(\pi_{\fra} (\mu), V_{\fra} (\mu) )$ occurring in the continuous spectrum of  $L^2 (\Gamma \backslash G)$, we have
	\begin{equation}\label{eq: spectral side, Eisenstein}
	W_{F_{\pi_{\fra } (\mu) } (P_f),\, \psi_2 } (1) = c_{\Gamma} (\pi_{\fra} (\mu) ; \psi_1, \psi_2) 
	\int_{\BCx} \nu (z)  \widehat \eta \lp \frac 1 z \rp \EJ_{\pi_{\fra} (\mu), \, \psi_1} \lp \frac {\kappa} z \rp d^\times z,
	\end{equation}
	with
	\begin{equation*}
	c_{\Gamma} (\pi_{\fra} (\mu) ; \psi_1, \psi_2) = \frac {c(\pi_{\fra} (\mu) ; \psi_2/\psi_1) c_{\psi_1}( \pi_{\fra} (\mu), \Gamma ) } {4 \pi^2 }.
	\end{equation*}
\end{lem}

\begin{proof}
	First, in view of \eqref{eq: second constant}, we have
	\begin{equation*}%\label{eq: spectral proof, 1}
	W_{F_{\pi} (P_f),\, \psi_2 } (1) = c (\pi; \psi_2/ \psi_1) W_{ F_{\pi} (P_f),\, \psi_1 } \begin{pmatrix}
	\kappa & \\
	 & 1
	\end{pmatrix}.
	\end{equation*}
	So it is enough to compute $W_{ F_{\pi} (P_f),\, \psi_1 } \begin{pmatrix}
	\kappa & \\
	& 1
	\end{pmatrix}$. 
	Let us now drop $\psi_1$ from notations. 
	By  \eqref{eq: def of F (phi)}, \eqref{eq: inner product identity} and \eqref{eq: first constant}, we see that $W_{ F_{\pi} (P_f)  } \begin{pmatrix}
	a & \\
	& 1
	\end{pmatrix}$ is characterized by
	\begin{equation*}
    \int_{\BCx} W_{F_{\pi} (P_f) } \begin{pmatrix}
    a & \\ & 1
    \end{pmatrix} \overline {W_{\varphi } } \begin{pmatrix}
    	a & \\ & 1
    	\end{pmatrix}  d^\times a = c  (\pi, \Gamma) \int_{ N \backslash G} f (g) \overline {W_{\varphi} (g)} d g \hskip 10 pt \text{ all } \varphi \in V_{\pi}^{\infty}.
	\end{equation*}
	Now substitute the definition of $f (g) = f_{\eta, \nu} (g)$ on the right and integrate over the open Bruhat cell. We get
	\begin{align*}
	\int_{ N \backslash G} f (g) \overline {W_{\varphi} (g)} d g = \frac 1 {4 \pi^2} \int_{\BC } \int_{ \BCx } \nu (z)  \eta (u) \overline {W_{\varphi} } \lp \varw \begin{pmatrix}
		1 & u \\
		 & 1
		\end{pmatrix} 
		\begin{pmatrix}
		z & \\
		 & 1
		\end{pmatrix}\rp   d^{\times} z \, d u.
	\end{align*}
	We now use the identity \eqref{eq: Bessel function, Weyl element} for the Bessel function.
	\begin{align*}
	W_{\varphi}  \lp \varw \begin{pmatrix}
	1 & u \\
	& 1
	\end{pmatrix} 
	\begin{pmatrix}
	z & \\
	& 1
	\end{pmatrix}\rp 
& = W_{\pi { \tiny \begin{pmatrix}
		z & u \\
		 & 1
		\end{pmatrix}} \varphi } (\varw) \\
& = \int_{\BCx} \EJ_{\pi} (a) W_{\pi { \tiny \begin{pmatrix}
		z & u \\
		& 1
		\end{pmatrix}} \varphi }
\begin{pmatrix}
a & \\
& 1
\end{pmatrix} d^\times a \\
& = \int_{\BCx} \psi_1 (a u) \EJ_{\pi} (a) W_{\varphi} \begin{pmatrix}
a z & \\
& 1
\end{pmatrix} d^\times a\\
& = \int_{\BCx} \psi_2 \lp  \frac 1 {\kappa} \frac {a u} {z} \rp \EJ_{\pi} \lp \frac {a} {z} \rp W_{\varphi} \begin{pmatrix}
a  & \\
& 1
\end{pmatrix} d^\times a.
	\end{align*}
Therefore,
\begin{align*}
\int_{ N \backslash G} f (g) &   \overline {W_{\varphi} (g)} d g  = \\
& \frac 1 {4 \pi^2} \int_{\BC }  \int_{\BCx} \int_{ \BCx } \nu (z) \eta (u)  \psi_2\- \lp   \frac {a u} {\kappa z} \rp   { \EJ_{\pi} \lp \frac {a} {z} \rp} \overline { W_{\varphi} } \begin{pmatrix}
a  & \\
& 1
\end{pmatrix}  d^\times a \,  d^{\times} z \, d u.
\end{align*}
Interchanging the order of integrations, this becomes
\begin{align*}
\int_{\BCx} \left\{  \frac 1 {4 \pi^2}  \int_{ \BCx } \nu (z)  \widehat \eta \lp \frac {a} {\kappa z} \rp    { \EJ_{\pi} \lp \frac {a} {z} \rp}  d^{\times} z \right\} \overline { W_{\varphi} } \begin{pmatrix}
	a  & \\
	& 1
	\end{pmatrix}   d^\times a.
\end{align*}
Note that $\widehat \eta$ is the Fourier transform of $\eta$ with respect to $\psi_2$. This yields the identity
\begin{equation*}
W_{ F_{\pi} (P_f)  } \begin{pmatrix}
a & \\
& 1
\end{pmatrix} = 
\frac {c  (\pi, \Gamma) } {4 \pi^2}  \int_{ \BCx } \nu (z)  \widehat \eta \lp \frac {a} {\kappa z} \rp   { \EJ_{\pi} \lp \frac {a} {z} \rp}  d^{\times} z,
\end{equation*}
and then follows the formula \eqref{eq: spectral side, discrete}. 

Proceeding  exactly as above, we may derive \eqref{eq: spectral side, Eisenstein}.
\end{proof}

\subsection{The Kloosterman-Spectral Formula}

In conclusion, as a consequence of Theorem \ref{thm: Whittaker} applied to the Poincar\'e
series $P_f$ with $f = f_{\eta,\nu}$ as in \S \ref{sec: choice of f}, we have the following Kloosterman-spectral formula.

\begin{thm}\label{thm: first formula}
	Let $\eta \in \mathscr S (\BC)$ and $\nu \in C^\infty_c (\BCx)$. Set $A_{\eta, \nu} (z) = \nu (1/z) \widehat \eta (z)$. Let  $\kappa \in \BCx$ be such that $\psi_2 (z) = \psi_1 (\kappa z)$. Then
	\begin{align*}
	\sum_{ c\in \Omega (\Gamma) } & \frac {Kl_{\Gamma} (c; \psi_1, \psi_2 ) } {|c|^2} A_{\eta, \nu} \lp \frac 1 c \rp \\
	 = &\ \sum_{ \pi \in \Pi_d (\Gamma) } c_{\Gamma} (\pi; \psi_1, \psi_2) \int_{ \BCx} A_{\eta, \nu} (z) \EJ_{\pi, \, \psi_1} (\kappa z) d^\times z \\
	&\ + \sum_{ \fra} \frac 1 {4 \pi i} \frac 1 {m_{\fra} |\Lambda_{\fra}|} \int_{\Re \mu = 0}  c_{\Gamma} (\pi_{\fra} (\mu) ; \psi_1, \psi_2) \left\{  \int_{ \BCx} A_{\eta, \nu} (z) \EJ_{\pi_{\fra} (\mu) ,\, \psi_1} (\kappa z) d^\times z \right\} d \mu, 
	\end{align*}
 with $\mu \in \mathfrak X_{m_{\fra}} (\BCx)$.
\end{thm}

\section{An Explicit Kloosterman-Spectral Formula - the Kuznetsov Trace Formula}

In this section, we shall express the Bessel functions in terms of the classical Bessel functions and relate the constants $c_{\Gamma} (\pi; \psi_1, \psi_2)$ to classical quantities.

\subsection{ }

%The Bessel functions $\EJ_{\pi, \, \psi} $ depend only upon the isomorphism class of $\pi$ and the additive character $\psi$ used to define the Kirillov model $\mathcal K (\pi, \psi)$. 
By \cite[Proposition 18.5]{Qi-Bessel}, the Bessel functions $\EJ_{\pi, \, \psi} $ may be expressed explicitly  in terms of the classical Bessel functions. For $\pi \cong \pi_d (s)$, with either $s = i t$ purely imaginary or $ s = t$ on the segment $\lp 0, \frac 1 2 \rp$, and $\psi = \psi_{a}$, with $a \in \BCx$, we have
\begin{align*}
\EJ_{\pi, \, \psi } (z)  =  \frac {2 \pi^2 } {\sin (2\pi s)}  |a^2 z| \lp J_{ s, 2 d} (4 \pi  a \sqrt z) -  J_{- s, - 2 d} (4 \pi  a \sqrt z) \rp,
\end{align*}
where $%\label{7def: J mu m (z), n=2, C}
	J_{s,   2 d} (z) = J_{- 2 s - d } \lp  z \rp J_{- 2 s + d  } \lp  {\overline z} \rp.
	$ Here $\sqrt z$ is the principal branch of the square root of $z$ and the expression on the right   is independent on the argument of $z$ modulo $2 \pi$.

%\subsection{Representations of $K$}

\subsection{ } \label{sec: compute the constants} We first consider the discrete component $(\pi, V_{\pi}) \subset L^2(\Gamma \backslash G)$ with $\pi \cong \pi_d (s)$, $s = i t$ or $s = t$. The restriction of $V_{\pi} $ on $K$ decomposes into the direct sum of the $(2 l + 1)$-dimensional representations $ \sigma_l $ of $K$ with $l \geqslant |d|$. Furthermore, each representation $\sigma_l$ decomposes into the direct sum of  the one-dimensional weight spaces of weight $q$ with $|q| \leqslant l$. Accordingly, we say that a vector of $V_{\pi}^\infty $ is of $K$-type $(l, q) $ if it lies in the $q$-weight space of $\sigma_l$.\footnote{By convention, the weights $q$ are the eigenvalues for $H = - \frac 1 2 i \otimes \begin{pmatrix}
	i & \\
	& - i
	\end{pmatrix}$ as an element in the complexified Lie algebra $\BC \otimes_{\BR} \mathfrak {su}(2)$. Note that the definition of $K$-type coincides with that in \cite{B-Mo} if  $q$ is changed into $-q$.} Therefore, $V_{\pi}^\infty$ contains a unique vector $\varphi$ of type $(|d|, d) $  normalized such that $\varphi_{\pi}$ has Petersson square norm $1$. %$(\varphi_{\pi}, \varphi_{\pi})_{\pi} = 1 $
It is known that $\varphi_{\pi} (g)$ will have a Fourier expansion at   infinity in terms of Jacquet's Whittaker function $W_{\omega}^{s, d}$ (see, for example, \cite{J-L}),
\begin{equation*}
\varphi_{\pi} \lp 
\begin{pmatrix}
1 & z \\
  & 1
\end{pmatrix} 
\begin{pmatrix}
r & \\
   & 1
\end{pmatrix} 
k \rp 
= \sum_{\omega \in \Lambda_{\infty}' } a_{\omega} (\varphi_\pi) \psi_{\omega}   ( z) W_{\omega}^{s, d} ( r, k),
\end{equation*}
with $z \in \BC$, $r \in \BR_+$, $k \in K$. %Here $W_{\omega}^{s, d}$ is Jacquet's Whittaker function that we are about to define. 
%We stress that up to constant multiple the coefficients $a_{\omega} (\varphi_{\pi})$ does not depend on the vector in $V_{\pi}^\infty$ that we choose.
Now, to define $W_{\omega}^{s, d}$, we first introduce the function
$$\phi_{s, d} \lp 
\begin{pmatrix}
1 & z \\
& 1
\end{pmatrix} 
\begin{pmatrix}
r & \\
& 1
\end{pmatrix}
k_{v,\, w} \rp = \begin{cases}
r^{ 2 s + 1} v^{2 d}, & \text{ if } d \geq 0, \\
r^{ 2 s + 1} \overline v^{\, - 2 d}, & \text{ if } d < 0.
\end{cases} $$ 
It is readily verified that $\phi_{s, d}  $ is left $  \mu_{s, d} \delta_{\infty}^{\frac 1 2}$-variant, that is, $\phi_{s, d} (\underline {u} \hskip 1 pt \underline {a} g ) = |a|^{ 2 s+1 } [a]^{d} \phi_{s, d} (g) $, with $ \underline u \in N$ and $\underline a \in A$.
For $\Re s > 0$, $W_{\omega}^{s, d} (r, k)$ is given by
\begin{equation*}
W_{\omega}^{s, d} (r, k) = \int_{N} \psi_{\omega}\- (n) \phi_{s, d} (\varw n \underline {r} k) d n,
\end{equation*}
in which the integral converges absolutely, and $W_{\omega}^{s, d} (r, k)$ is defined for all $ s$ via meromorphic continuation, except for the poles that occur in the case $ \omega = 0$. 

For computing the constant $c_{\Gamma} (\pi; \psi_1, \psi_2)$, we shall be concerned only with those nonzero $\omega $ and the values of $W_{\omega}^{s, d} (r, k_{v, 0})$, with $v = e^{\frac 1 2 i \phi}$. With these in mind, we have the following lemma.

\begin{lem} \label{lem: Jacquet's Whittaker function} 
Let $\phi \in \BR/2\pi \BZ$ and $r \in \BR_+$. Set $v = e^{\frac 1 2 i \phi}$ and $a = r e^{i \phi}$. Then, for $\omega \in \Lambda_{\infty}' \smallsetminus \{0\}$, we have
	\begin{equation}\label{eq: Jacquet, omega neq 0}
	W_{\omega}^{s, d} (r, k_{v, 0}) = 
	\frac {2 (-1)^d (2 \pi)^{2 s + |d| + 1}  } {\Gamma (2 s + |d|+1 ) }   |\omega|^{2 s+|d|} |a|^{|d|+1} [ \omega^2 a]^d K_{|d| - 2 s} (4 \pi |\omega a | ).
	\end{equation}
	 
\end{lem}

\begin{proof}
	On expressing $\varw \underline {z} \hskip 1 pt \underline {r} \hskip 1 pt k_{v, 0} $ in the Iwasawa coordinates, changing the variables from $z$ to $r z$ and using the polar coordinates for $z$, we arrive at
	\begin{align*}
	W_{\omega}^{s, d} (r, k_{v, 0}) 
	= 2 r^{1-2 s} e^{i d \phi} \int_{0}^\infty   \frac {x^{2|d| + 1} I_{\omega}^d (r x)}   { (1 + x^2)^{ 2 s + |d| + 1} }    d x,
	\end{align*}
	with
	\begin{equation*}
	I_{\omega}^d (x ) = \int_0^{2 \pi}  e^{ 2 \pi i \Tr (\omega  x e^{i \theta })- 2  i d \theta } d \theta.
	\end{equation*}
	%For $\omega = 0$, we prove \eqref{eq: Jacquet, omega = 0} by the observation
	%\begin{equation*}
	%I_{0}^d (x ) = 
	%\begin{cases}
	%1, & \text{ if } d = 0, \\
	%0, & \text{ if } d \neq 0.
	%\end{cases}
	%\end{equation*}
    We first recall the integral representation of Bessel,
	\begin{equation*}%\label{6eq: Bessel integral}
	2 \pi i^m J_{m} ( x) = \int_{0}^{2 \pi} e^{ i x \cos \theta + i m \theta} d \theta, \hskip 10 pt m \in \BZ,
	\end{equation*}
	then we find that
	$$ I^d_{\omega} (x) =  2 \pi  (-1)^d [\omega]^{2 d } J_{2|d|} (4\pi  |\omega| x).$$
	Consequently,
	\begin{equation*}
	W_{\omega}^{s, d} (r, k_{v, 0}) =   4 \pi (-1)^d  [\omega]^{2 d} r^{1-2 s} e^{i d \phi} \int_0^\infty \frac {x^{2|d| + 1} J_{2|d|} (4\pi  |\omega| r x)}   { (1 + x^2)^{ 2 s + |d| + 1} }    d x.
	\end{equation*}
	Using the formula \cite[{\bf 6.565} 4]{G-R}
	\begin{equation*}
	\int_0^\infty \frac {x^{\nu+1} J_{\nu} (a x)  } {(1+x^2)^{\mu+1}}  d x = \frac {a^{\mu}} {2^{\mu} \Gamma (\mu + 1) } K_{\nu- \mu} (a), \hskip 10 pt 2 \Re  \mu + \frac 3 2 > \Re \nu > -1 ,\ a > 0,
	\end{equation*}
	we obtain
	\begin{equation*} 
	W_{\omega}^{s, d} (r, k_{v, 0})   = 
	\frac { 4 \pi (-1)^d    (2 \pi |\omega| )^{2 s+|d|} [\omega]^{2 d }  r^{|d|+1 } e^{i d \phi}} {\Gamma (2 s + |d|+1 ) } K_{|d| - 2 s} (4 \pi |\omega | r ).
	\end{equation*}
	Therefore \eqref{eq: Jacquet, omega neq 0} is proven for $\Re s > 0$ and remains valid by analytic continuation.
	
\end{proof}

It will be convenient to work with renormalized Fourier coefficients
\begin{equation*}
A_{\omega} (\varphi_{\pi}) = (2\pi)^{2 s} |\omega|^{2 s} [\omega]^d a_{\omega}  (\varphi_{\pi}) .
\end{equation*}
Then if $\psi_i (z) = \psi_{\omega_i} (z)$, $i = 1, 2$, we see that
\begin{equation*}
W_{\varphi_\pi,\, \psi_i} 
\begin{pmatrix}
a & \\
  & 1
\end{pmatrix}
=   \frac {2 |\Lambda_{\infty}| (-1)^d (2 \pi)^{ |d| + 1}  } {\Gamma (2 s + |d|+1 ) }  A_{\omega_i} (\varphi_{\pi}) |a| |\omega_i a|^{ |d|}  [ \omega_i a]^d   K_{|d| - 2 s} (4 \pi |\omega_i a | ).
\end{equation*}
%and if $\psi_2 (z) = \psi_{\omega_2} (z)$ then a similar formula holds.

Let us first compute $c (\pi; \psi_2/\psi_1)$, which is defined by
\begin{equation*}
W_{\varphi_\pi,\, \psi_2} 
\begin{pmatrix}
\omega_1 & \\
 & 1
\end{pmatrix}
= c (\pi; \psi_2/\psi_1)
W_{\varphi_\pi,\, \psi_1}  
\begin{pmatrix}
\omega_2  & \\
& 1
\end{pmatrix}.
\end{equation*}
It follows that
\begin{equation*}
c (\pi; \psi_2/\psi_1) = \frac {A_{\omega_2} (\varphi_{\pi}) |\omega_1|  } {A_{\omega_1}(\varphi_{\pi}) |\omega_2|  }.
\end{equation*}

On the other hand, as $\varphi_{\pi}$ has Petersson square norm $1$, $c_{\psi_1} (\pi, \Gamma)$ is defined by
\begin{equation*}
c_{\psi_1} (\pi, \Gamma) = \int_{ \BCx} W_{\varphi_\pi,\, \psi_1} 
\begin{pmatrix}
a & \\
& 1
\end{pmatrix}
\overline { W_{\varphi_\pi,\, \psi_1} }
\begin{pmatrix}
a & \\
& 1
\end{pmatrix} 
d^\times a.
\end{equation*}
We first let $s = it$. Integrating in the polar coordinates, we obtain
\begin{align*}
c_{\psi_1} (\pi, \Gamma) =     \frac {8 |\Lambda_{\infty}|^2 (2\pi)^{2 |d| + 3  } } {|\Gamma (2 i t + |d| + 1)|^2 } & |A_{\omega_1} (\varphi)|^2 |\omega_1|^{2|d|} \\
& \int_{0}^\infty r^{2 |d|+1} K_{|d| - 2 i t} (4 \pi |\omega_1| r) K_{|d| + 2 i t} (4 \pi |\omega_1| r)  d  r.
\end{align*}
As a special case of the Weber-Schafheitlin formula for $K$-Bessel functions (see \cite[{\bf 6.576} 4]{G-R}), we have
\begin{align*}
\int_0^\infty & x^{\, \rho - 1} K_{\nu} (a x) K_{\mu} (a x) d x \\
& =   \frac 
{2^{\, \rho-3}\Gamma \lp \frac 1 2 \lp  \rho + \nu + \mu \rp \rp
\Gamma \lp \frac 1 2 \lp   \rho + \nu - \mu \rp \rp
\Gamma \lp \frac 1 2 \lp   \rho - \nu + \mu \rp \rp
\Gamma \lp \frac 1 2 \lp   \rho - \nu - \mu \rp \rp}
{a^{\, \rho}\Gamma (\rho)}, \\
& \hskip 205 pt   \Re \rho > |\Re \nu | + |\Re \mu |, \ a > 0.
\end{align*}
This yields
\begin{equation*}
c_{\psi_1} (\pi, \Gamma) = 
\frac {2 \pi |\Lambda_{\infty}|^2 |A_{\omega_1} (\varphi)|^2} 
{(2|d|+1) |\omega_1|^2 }.
\end{equation*}
Therefore, we get
\begin{equation*}
c_{\Gamma} (\pi; \psi_1, \psi_2) = \frac {|\Lambda_{\infty}|^2 A_{\omega_2}(\varphi_{\pi}) \overline {A_{\omega_1 } (\varphi_{\pi})}   } {2 \pi (2|d|+1) |\omega_1 \omega_2| }.
\end{equation*}
Similarly, when $s = t$ and $d = 0$, with $t \in \lp 0, \frac 1 2 \rp$, we have
\begin{align*}
c_{\psi_1} (\pi, \Gamma) 
= \frac {2 \pi |\Lambda_{\infty}|^2 |A_{\omega_1} (\varphi)|^2  } {   G_{t, 0} |\omega_1|^{ 2}  }   ,
\end{align*}
with 
\begin{equation*}
G_{t, 0} =   \frac  { \Gamma (1+2  t   ) } {  \Gamma (1 -2 t  )   }.
\end{equation*}
Hence 
\begin{equation*}
c_{\Gamma} (\pi; \psi_1, \psi_2) = \frac {A_{\omega_2}(\varphi_{\pi}) \overline {A_{\omega_1 } (\varphi_{\pi})}   } {2 \pi  G_{t, 0} |\omega_1 \omega_2| }.
\end{equation*}
Combining these, if we simply put $G_{it, d} = 1$, then
\begin{equation*}
c_{\Gamma} (\pi; \psi_1, \psi_2) = \frac {|\Lambda_{\infty}|^2 A_{\omega_2}(\varphi_{\pi}) \overline {A_{\omega_1 } (\varphi_{\pi})}   } {2 \pi (2|d|+1) G_{s, d} |\omega_1 \omega_2| }.
\end{equation*}

\subsection{ } \label{sec: compute constants for Eisenstein series} For the continuous spectrum, we fix a cusp $\fra$ and let $\mu (z) = \mu_{s, d} (z) = [z]^d |z|^{2 s}$, with $s = i t$ and $m_{\fra} | d$.
Choose $f \in V_{\fra} (\mu)^{\infty}$ to be
\begin{equation*}
f (g) = \phi_{ s, d} (g_{\fra} g).
\end{equation*}
We have $\langle f, f \rangle_{\fra} = 1/ (2|d|+1)$. With $f$ is associated the Eisenstein series $E_{\fra} (g; f, \mu)$. To describe the Fourier coefficients of this Eisenstein series we need some notations. 

For each cusp $\fra $ let $\Omega_{\fra} (\Gamma) = \left\{ c \in \BCx : g_{\fra} \Gamma \cap N \varw \underline {c} N \neq \O , \, \underline c \in A \right\} $, and for each $c \in \Omega_{\fra} (\Gamma) $ define $\Gamma_c^{\fra} = g_{\fra} \Gamma \cap N \varw \underline c N $. Then $\Gamma^{\fra}_c$
is left invariant under $g_{\fra} \Gamma_{\fra} g_{\fra}\-$ and right invariant under $\Gamma_{\infty}$. For each $\gamma \in \Gamma^{\fra}_c$ we write
\begin{equation*}
\gamma = n_1 (\gamma) \varw \underline c n_2 (\gamma),
\end{equation*}
with $n_1 (\gamma), n_2 (\gamma) \in N$. 

For $c \in \Omega_{\fra} (\Gamma) $ we define   the Ramanujan sum $Kl^{\fra}_{\Gamma} (c; 1, \psi) $ by
\begin{equation*}%\label{eq: Kloosterman sum 2}
Kl^{\fra}_{\Gamma} (c; 1, \psi) =  \sum_{\gamma \in  g_{\fra} \Gamma_{\fra} g_{\fra}\- \backslash \Gamma^{\fra}_c / \Gamma_{\infty} } \psi (n_2 (\gamma)),
\end{equation*}
and for $\omega \in \Lambda_{\infty}'$ we define the Kloosterman-Selberg zeta series
\begin{equation*}%\label{eq: Kloosterman zeta series}
Z_{\Gamma}^{\fra} (\mu ; 0, \omega) = \sum_{c \in \Omega_{\fra} (\Gamma) }  \frac { Kl^{\fra}_{\Gamma} (c; 1, \psi_{\omega}) } {\mu  (c) \delta_{\infty} (c)^{\frac 1 2}},
\end{equation*}
if $\Re \mu $ is sufficiently large. Recall that $ \delta_{\infty} (c) = |c|^2 $ is the modulus character. These series $Z_{\Gamma}^{\fra} (\mu ; 0, \omega)$ have meromorphic continuation to the whole complex plane and have no poles on the line $\Re \mu = 0$.

We may now describe the Fourier expansion of the Eisenstein series $E_{\fra} (g ; f, \mu)$. For 
$ g = \begin{pmatrix}
	1 & z \\
	& 1
\end{pmatrix} 
\begin{pmatrix}
	r & \\
	& 1
\end{pmatrix} 
k$, we have
\begin{align*}
E_{\fra} (g ; f, \mu) = \delta_{\fra \infty} m_{\infty} \phi_{s, d} (   g)
& + \frac 1 {|\Lambda_{\infty}| } Z_{\Gamma}^{\fra} (\mu ; 0, 0) W^{s, d}_0 (r, k ) \\
& + \frac 1 {|\Lambda_{\infty}| } \sum_{ \omega \in \Lambda_{\infty}' \smallsetminus \{0\} } \psi_{\omega} (z) Z_{\Gamma}^{\fra} (\mu ; 0, \omega) W^{s, d}_{\omega} (r, k ),
\end{align*}
where the first term exists only if $\fra = \infty$. %in which $A_{\infty} (\Gamma) =    \left\{ \underline a \in A : a \in \eta_{m_{\infty} } \right\} = A \cap \Gamma_{\infty}' \cong \Gamma_{\infty}   \backslash   \Gamma_{\infty}'$.

Let $\psi_1   = \psi_{\omega_1}  $ and $\psi_2  = \psi_{\omega_2}  $, with $\omega_1, \omega_2 \in \Lambda_{\infty}' \smallsetminus \{0\}$. From Lemma \ref{lem: Jacquet's Whittaker function}, for $i = 1, 2$, we have 
\begin{align*}
W_{E_{\fra} ( \cdot ; f,\, \mu),\, \psi_i} 
\begin{pmatrix}
a & \\
& 1
\end{pmatrix} = & \frac {2 (-1)^d (2 \pi)^{2 s + |d| + 1}  } {  \Gamma (2 s + |d|+1 ) } \\
& Z_{\Gamma}^{\fra} (\mu ; 0, \omega_i) |\omega_i|^{2 s+|d|} |a|^{|d|+1} [ \omega_i^2 a]^d K_{|d| - 2 s} (4 \pi |\omega_i a | ).
\end{align*}
%as well as  a similar formula for $\psi_2 $. 
Then we get
\begin{equation*}
c (\pi_{\fra} (\mu) ; \psi_2 / \psi_1 ) = 
\frac { Z_{\Gamma}^{\fra} (\mu ; 0, \omega_2) |\omega_2|^{2 s - 1} [\omega_2]^d } {Z_{\Gamma}^{\fra} (\mu ; 0, \omega_1)  |\omega_1|^{2 s - 1} [\omega_1]^d }
\end{equation*}
and
\begin{equation*}
c_{\psi_1} (\pi_{\fra} (\mu), \Gamma) = 
\frac {2 \pi |Z_{\Gamma}^{\fra} (\mu ; 0, \omega_1)|^2 } {  |\omega_1|^2 }.
\end{equation*}
Therefore,
\begin{equation*}
c_{\Gamma} (\pi_{\fra} (\mu) ; \psi_1, \psi_2) = \frac   {Z_{\Gamma}^{\fra} (\mu ; 0, \omega_2) \overline {Z_{\Gamma}^{\fra} (\mu ; 0, \omega_1) }   } {2 \pi     |\omega_1 \omega_2| } \mu \lp \frac {\omega_2} {\omega_1} \rp.
\end{equation*}

\subsection{ }

In conclusion, with the computations above, the Kuznetsov trace formula in Theorem \ref{thm: main} then follows from Theorem \ref{thm: first formula}, with the choice of weight function $F (z) = A_{\eta, \nu} \lp z / \omega_1 \omega_2 \rp |z/ \omega_1 \omega_2|$. It should be remarked that $F (z)$ of this form actually exhaust all  functions in $C^{\infty}_c (\BCx)$ if one let $\eta$ and $\nu$ vary.

\bibliographystyle{alphanum}

\bibliography{references}

\end{document}